\documentclass[11pt,a4paper]{article}
\usepackage[margin=1in]{geometry}
\usepackage[utf8]{inputenc}
\usepackage[T1]{fontenc}
\usepackage{amsmath,amssymb,amsthm,mathtools}
\usepackage{enumitem}
\usepackage{hyperref}
\usepackage{booktabs}
\usepackage{microtype}
\usepackage{lmodern} % <--- Add this line

\hypersetup{colorlinks=true,linkcolor=blue,citecolor=blue,urlcolor=blue}

\newtheorem{theorem}{Theorem}[section]
\newtheorem{proposition}[theorem]{Proposition}
\newtheorem{lemma}[theorem]{Lemma}

\theoremstyle{definition}
\newtheorem{definition}[theorem]{Definition}
\newtheorem{assumption}[theorem]{Assumption}
\theoremstyle{remark}
\newtheorem{remark}[theorem]{Remark}

\newcommand{\R}{\mathbb{R}}
\newcommand{\Sph}{\mathbb{S}}
\newcommand{\dist}{\mathrm{dist}}
\newcommand{\diverg}{\mathrm{div}}
\newcommand{\curl}{\mathrm{curl}}
\newcommand{\Rfoc}{R_{\mathrm{foc}}}

\newcommand{\eps}{\varepsilon}
\newcommand{\dH}{d_{\mathcal{H}}}

\title{A Geometric Framework for Conditional Stability\\ in Overdetermined Elliptic Problems}
\author{M. Barkatou \and M. El Morsalani}
\date{}

\begin{document}
\maketitle

\begin{abstract}
We present a geometric framework for studying stability and symmetry in overdetermined elliptic boundary value problems. The approach is based on a normal radial parametrization of domains with respect to a strictly convex core \(C\). We establish \emph{unconditional qualitative} stability for second- and fourth-order problems (Serrin and Bennett--Dalmasso). We then derive \emph{conditional quantitative} stability estimates assuming a residual estimate for the tangential component (Assumption~\ref{ass:residual}). The method does not rely on moving planes or the maximum principle. The quantitative results are conditional upon a sufficient analytical assumption that we carefully formulate, including a control of the tangential gradient; we verify this assumption in two model cases (Proposition~\ref{prop:C0} and Proposition~\ref{prop:annulus}). 
\emph{The quantitative results (Proposition~\ref{prop:weak} and Proposition~\ref{prop:strong}) are stated and proved in the model case \(C=\{0\}\); the general strictly convex core is discussed in Remark~\ref{rem:general-core-weak} and left for future work.}
\end{abstract}

\tableofcontents

% ===========================================================================
\section{Introduction}\label{sec:intro}
% ===========================================================================

Overdetermined boundary value problems are central in geometric analysis and mathematical physics. A classical result due to Serrin \cite{Serrin1971} states that if a bounded domain admits a solution to
\begin{equation}\label{eq:serrin}
  -\Delta u = \text{const}, \quad u = 0, \quad \partial_\nu u = \text{const} \quad \text{on } \partial\Omega,
\end{equation}
then \(\Omega\) must be a ball. Extensions to fourth-order equations were obtained by Bennett \cite{Bennett1986} and Dalmasso \cite{Dalmasso1990} for the problem
\begin{equation}\label{eq:biharm}
  \Delta^2 u = 1 \text{ in } \Omega, \quad u = \partial_\nu u = 0 \text{ on } \partial\Omega, \quad \Delta u = c \text{ on } \partial\Omega,
\end{equation}
where again the domain must be a ball. The radius of the ball is \(R = \sqrt{N(N+2)c}\).

In recent years, quantitative versions of such results have attracted attention. Stability estimates have been obtained for Serrin's problem \cite{MagnaniniPoggesi2020,BrascoMagnaniniPoggesi2020,CiraoloVezzoni2017} and for other overdetermined systems \cite{Barkatou2002,Barkatou2026}. These typically rely on integral identities, spectral gap inequalities, or moving plane methods.

In this work we develop a geometric framework that:
\begin{itemize}
  \item uses a normal radial parametrization with respect to a fixed strictly convex core \(C\);
  \item reduces the PDE to a family of one-dimensional problems along normal rays, with a controlled tangential residual;
  \item provides \emph{unconditional qualitative} compactness (via the \(C\)-GNP class) and \emph{conditional quantitative} estimates (via ODE comparison) under a sufficient analytical assumption (Assumption~\ref{ass:residual}).
\end{itemize}

The novelty of our approach lies in the geometric interpretation: the \(C\)-GNP class allows a transformation of the elliptic symmetry problem into a family of ODEs with a tangential residual. This is a departure from the classical moving plane or shape derivative methods, and it opens the door to quantitative stability results through ODE comparison and compactness arguments.

\medskip\noindent\textbf{Status of the quantitative results.} All quantitative results in Sections~\ref{sec:weak} and~\ref{sec:strong} are conditional upon Assumption~\ref{ass:residual}, which is stated explicitly and for which we provide a discussion of its validity and a verification in two model cases (Proposition~\ref{prop:C0} and Proposition~\ref{prop:annulus}). 
In particular, we emphasize that a control of the tangential gradient is necessary and is part of the assumption. 
In the \(C\)-NR framework, Assumption~\ref{ass:residual} is verified only for \(C=\{0\}\) (Proposition~\ref{prop:C0}). The annular model problem (Proposition~\ref{prop:annulus}) is a verification for a different problem, not for the \(C\)-NR problem with \(C=B_\rho\) (where \(B_\rho \subset \Omega\)); it serves as a guide for future work. The extension to a general strictly convex core is discussed in Section~\ref{sec:perspectives}.

\begin{remark}[Sign conventions for the biharmonic problem]\label{rem:signs}
Throughout, we use the convention
\[
  v := -\Delta u .
\]
With this choice, problem \eqref{eq:biharm} is equivalent to the coupled system
\begin{equation}\label{eq:biharm-system}
  -\Delta v = 1 \ \text{in } \Omega, \qquad -\Delta u = v \ \text{in } \Omega, \qquad u = \partial_\nu u = 0, \quad v = -c \ \text{on } \partial\Omega .
\end{equation}
The radial reference solution on a ball of radius \(R\) is
\[
  u(r) = \frac{(R^2-r^2)^2}{8N(N+2)}, \qquad v(r) = -\Delta u(r) = \frac{R^2}{2(N+2)} - \frac{r^2}{2N},
\]
so that \(\Delta u(R) = R^2/(N(N+2)) = c\) and \(v(R) = -c\). The choice \(v = +\Delta u\) (used in an earlier version) leads to a system with opposite signs and is avoided here.
\end{remark}

\begin{remark}[Relation with the maximal surface equation]\label{rem:maximal}
The framework developed here is tailored to linear operators (\(\Delta\) and \(\Delta^2\)). A closely related but distinct setting is the maximal surface equation
\[
  \diverg\!\left( \frac{D u}{\sqrt{1-|Du|^2}} \right) = 0,
\]
for which unconditional rigidity results have been obtained in \cite{Lai2026}. The normal radial parametrization and the ODE reduction used here admit a quasilinear analogue, but the tangential residual then becomes nonlinear; we leave this extension for future work (see Section~\ref{sec:perspectives}).
\end{remark}

% ===========================================================================
\section{Geometric framework: normal radial domains}\label{sec:geometric}
% ===========================================================================

Let \(C \subset \R^N\) be a compact, \emph{strictly convex} set with \(C^2\) boundary and nonempty interior. Denote by \(\nu(c)\) the outward unit normal to \(\partial C\) at \(c\), and by \(\kappa_1(c),\dots,\kappa_{N-1}(c) > 0\) the principal curvatures of \(\partial C\) at \(c\).

\subsection{\(C\)-NR and \(C\)-GNP domains}\label{sec:CNR}

\begin{definition}[\(C\)-NR domains]\label{def:CNR}
A bounded domain \(\Omega \supset C\) is said to be \emph{normal radial with respect to \(C\)} (or \(C\)-NR) if every \(x \in \Omega\setminus C\) can be uniquely written as
\begin{equation}\label{eq:param}
  x = c + r\nu(c), \quad c \in \partial C, \quad 0 < r < d(c),
\end{equation}
for some positive function \(d \in C^2(\partial C)\), called the \emph{thickness function}. The mapping \(\Phi(c,r) = c + r\nu(c)\) is a \(C^1\)-diffeomorphism from \(\partial C \times (0,d(c))\) onto \(\Omega\setminus C\).
\end{definition}

\begin{remark}[Jacobian and focal condition]\label{rem:jacobian}
For strictly convex \(C\), the normal rays do not intersect, and the Jacobian \(J(c,r) = \det D\Phi(c,r)\) is positive. In principal coordinates on \(\partial C\),
\[
  J(c,r) = \prod_{i=1}^{N-1} \bigl(1 + r\kappa_i(c)\bigr),
\]
and the focal condition
\[
  r < \Rfoc(C) := 1/\max_{c,i} \kappa_i(c)
\]
ensures that \(1+r\kappa_i(c) \ge c_0 > 0\) uniformly. In particular, \(J\) is uniformly bounded above and below by constants depending only on \(\partial C\), \(\delta\), and \(\|d\|_{C^1}\).
\end{remark}

\begin{remark}[Strict convexity]\label{rem:strict}
The strict convexity assumption on \(C\) (i.e.\ \(\kappa_i(c) > 0\) for all \(c,i\)) is used to guarantee \(J(c,r) > 0\). Without it, the formula for \(J\) in principal coordinates may involve a sign and \(J\) may vanish on flat parts of \(\partial C\). For a non-strictly convex core, the analysis requires additional care; we do not address this case here.
\end{remark}

Following \cite{Barkatou2002,Barkatou2026}, we introduce the class \(O_C\) (called \(C\)-GNP) of domains satisfying the geometric normal property: each inward normal ray from \(\partial\Omega\) meets \(C\). For such domains the thickness function is well defined and the parametrization \eqref{eq:param} holds.

\begin{assumption}[Uniform regularity and focal condition]\label{ass:uniform}
For the quantitative estimates, we assume that the thickness functions \(d\) are uniformly bounded in \(C^{2,\alpha}(\partial C)\) for some \(\alpha \in (0,1)\), and that there exists \(\delta > 0\) such that \(d(c) \ge \delta\) for all \(c \in \partial C\). Moreover, we impose the focal condition
\[
  \sup_{c \in \partial C}\ \sup_{0 \le r \le d(c)} r < \Rfoc(C).
\]
\end{assumption}

\subsection{Uniform cone property and extension operators}\label{sec:extension}

\begin{lemma}[\(H^1\) extension]\label{lem:extH1}
Assume \(\{\Omega_n\} \subset O_C\) with \(d_n\) uniformly bounded in \(C^1\) and \(d_n \ge \delta > 0\). Then there exist a fixed ball \(D\) containing all \(\Omega_n\) and a uniform constant \(C_e\) such that for every \(n\) there is a bounded linear extension operator \(E_n : H^1(\Omega_n) \to H^1(D)\) with \(\|E_n u\|_{H^1(D)} \le C_e \|u\|_{H^1(\Omega_n)}\).
\end{lemma}

\begin{proof}
The uniform \(C^1\) bounds on \(d_n\) imply that the boundaries \(\partial\Omega_n\) are uniformly Lipschitz and satisfy a uniform cone condition. Then the classical Stein extension theorem (see \cite{Stein1970}, or \cite{Grisvard1985,Chenais1975} for the parameter-dependent version) applies with uniform constants.
\end{proof}

\begin{lemma}[\(H^2\) extension]\label{lem:extH2}
Under Assumption~\ref{ass:uniform}, there exist a fixed ball \(D\) and a uniform constant \(C_{e,2}\) such that for every \(n\) there is a bounded linear extension operator \(E_n : H^2(\Omega_n) \to H^2(D)\) with \(\|E_n u\|_{H^2(D)} \le C_{e,2} \|u\|_{H^2(\Omega_n)}\).
\end{lemma}

\begin{proof}
By Assumption~\ref{ass:uniform}, \(d_n \in C^{2,\alpha}(\partial C)\) uniformly, so the boundaries \(\partial\Omega_n\) are uniformly \(C^{2,\alpha}\). The parametrization \(\Phi_n\) and its inverse have uniformly controlled \(C^2\) norms, and the uniform cone condition holds. The Calder\'on--Zygmund extension theorem for uniformly \(C^{2,\alpha}\) domains (see \cite{Stein1970, Grisvard1985}) then provides the uniform \(H^2\) extension operator.
\end{proof}

\begin{theorem}[Compactness of the \(C\)-GNP class]\label{thm:compactness}
Let \(\{\Omega_n\} \subset O_C\) with thickness functions \(d_n\) uniformly bounded in \(C^{2,\alpha}(\partial C)\) for some \(\alpha \in (0,1)\), and \(d_n \ge \delta > 0\). Then, up to a subsequence, \(\Omega_n \to \Omega_*\) in Hausdorff distance, with \(\Omega_* \in O_C\) and \(d_n \to d_*\) in \(C^{2,\alpha'}(\partial C)\) for every \(\alpha' < \alpha\).
\end{theorem}

\begin{proof}
By Arzel\`a--Ascoli applied to \(\{d_n\}\), there is a subsequence (not relabeled) with \(d_n \to d_*\) in \(C^{2,\alpha'}(\partial C)\) for every \(\alpha' < \alpha\). The parametrizations \(\Phi_n(c,r) = c + r\nu(c)\) converge uniformly to \(\Phi_*(c,r) = c + r\nu(c)\) on \(\partial C \times [0,M]\) and the domains \(\Omega_n = \Phi_n(\partial C \times (0,d_n))\) converge in Hausdorff distance to \(\Omega_* = \Phi_*(\partial C \times (0,d_*))\). The GNP property is stable under \(C^1\) convergence of the thickness functions, so \(\Omega_* \in O_C\).
\end{proof}

% ===========================================================================
\section{Reduction to one-dimensional problems with residual}\label{sec:reduction}
% ===========================================================================

For \(u \in C^2(\Omega)\), in normal coordinates we have the exact decomposition
\begin{equation}\label{eq:decomp}
  \Delta u = \frac1J \partial_r(J \partial_r u) + \Delta_{\Gamma_r} u,
\end{equation}
where \(\Delta_{\Gamma_r}\) is the Laplace--Beltrami operator on the parallel surface \(\Gamma_r = \{c + r\nu(c) : c \in \partial C\}\).

For the second-order problem \(-\Delta u = f_0(\dist(x,C))\), we define the radial profile \(\phi_c(r) = u(c + r\nu(c))\). Then \(\phi_c\) satisfies
\begin{equation}\label{eq:radial}
  -\frac1J \partial_r(J \partial_r \phi_c) = f_0(r) + \mathcal{R}_u(c,r),
\end{equation}
where the \emph{tangential residual} is \(\mathcal{R}_u(c,r) = -\Delta_{\Gamma_r} u(c + r\nu(c))\).

For the biharmonic system with \(v := -\Delta u\) (Remark~\ref{rem:signs}), setting \(v_c(r) = v(c+r\nu(c))\) and \(u_c(r) = u(c+r\nu(c))\), we get the coupled system
\begin{equation}\label{eq:coupled}
  -\frac1J \partial_r(J \partial_r v_c) = 1 + \mathcal{R}_v(c,r), \qquad
  -\frac1J \partial_r(J \partial_r u_c) = v_c + \mathcal{R}_u(c,r),
\end{equation}
with boundary conditions \(u_c(d(c)) = u'_c(d(c)) = 0\), \(v_c(d(c)) = -c\).

\subsection{Identification of radial and normal derivatives}\label{sec:identification}

\begin{proposition}\label{prop:radial-normal}
For any function \(w\) that depends only on \(r\) in normal coordinates (i.e.\ \(w\) is constant on each \(\Gamma_r\)), we have
\[
  \partial_r w = \partial_\nu w \quad \text{on } \partial\Omega.
\]
\end{proposition}

\begin{proof}
Let \(\Phi(c,r) = c + r\nu(c)\). By the focal condition, the unit normal to \(\Gamma_r\) is \(n_r = \nu(c)\) (up to orientation), a standard property of parallel hypersurfaces (see \cite[\S2]{Gray1990}). For a radial function \(w(r)\), \(\nabla w = w'(r)\nu(c)\), hence \(\partial_\nu w = \nabla w \cdot n_r = w'(r) = \partial_r w\).
\end{proof}

\subsection{Control of the tangential Laplacian}\label{sec:tangential}

\begin{lemma}\label{lem:tangential}
Under Assumption~\ref{ass:uniform}, there exists \(C>0\) depending only on the uniform bounds on \(d\), the curvatures of \(\partial C\), and the focal condition, such that for any smooth function \(w\),
\[
  \|\Delta_{\Gamma_r} w\|_{L^2(\Omega)} \le C \|\nabla_\tau w\|_{H^1(\Omega)}.
\]
\end{lemma}

\begin{proof}
For \(0<r<d(c)\), \(\Gamma_r\) is a compact smooth hypersurface without boundary, with curvatures uniformly controlled by those of \(\partial C\), by \(d\), and by the focal condition. In local coordinates on \(\partial C\), \(\Delta_{\Gamma_r} w = \frac{1}{\sqrt{|g|}}\partial_i(\sqrt{|g|} g^{ij}\partial_j w)\), with \(g_{ij}\) the induced metric on \(\Gamma_r\). The coefficients \(g^{ij}\) and their first derivatives are uniformly bounded. Elliptic regularity on the compact manifolds \(\Gamma_r\) (with uniform constants, see \cite{Hebey1996}) gives
\[
  \|\Delta_{\Gamma_r} w\|_{L^2(\Gamma_r)} \le C\bigl(\|\nabla_\tau w\|_{L^2(\Gamma_r)} + \|\nabla_\tau^2 w\|_{L^2(\Gamma_r)}\bigr),
\]
with \(C\) independent of \(r\). By Remark~\ref{rem:jacobian}, \(J\) is uniformly bounded above and below; integrating over \(r \in (0,d(c))\) with respect to \(J(c,r)\,dr\,d\sigma(c)\) yields the claim.
\end{proof}

\subsection{Analytical assumption on the tangential residual}\label{sec:assumption}

We formulate the following hypothesis, which is the key analytical ingredient for the quantitative estimates.

\begin{assumption}[Residual control]\label{ass:residual}
There exist constants \(C_1, C_2 > 0\), depending only on \(N\) and on the uniform bounds of the domain class, such that for any solution \((u,v)\) of the biharmonic system \eqref{eq:biharm-system} with \(v = -c + \eta\) on \(\partial\Omega\), and for the geometric perturbation parameter
\[
  \eps := \|d - R\|_{C^1(\partial C)},
\]
the boundary quantity
\begin{equation}\label{eq:epsb}
  \eps_b := \|\nabla_\tau v\|_{L^2(\partial\Omega)} + \|v + c\|_{L^2(\partial\Omega)}
\end{equation}
satisfies \(\eps_b \le C_1\eps\), and
\begin{equation}\label{eq:interior}
  \|\nabla_\tau u\|_{H^1(\Omega)}^2 + \|\nabla_\tau v\|_{H^1(\Omega)}^2 \le C_2\eps^2 .
\end{equation}
Consequently, by Lemma~\ref{lem:tangential},
\begin{equation}\label{eq:residual}
  \|\mathcal{R}_u\|_{L^2(\Omega)}^2 + \|\mathcal{R}_v\|_{L^2(\Omega)}^2 \le C_2'\eps^2 .
\end{equation}
\end{assumption}

\begin{remark}\label{rem:epsb}
The quantity \(\eps_b\) in \eqref{eq:epsb} includes the tangential gradient of \(v\) on the boundary. This is necessary for the ellipticity of the Laplace--Beltrami operator: indeed, \(\|\Delta_\Gamma w\|_{L^2}\) is controlled by \(\|\nabla_\tau w\|_{H^1}\), and the trace inequality gives control of \(\nabla_\tau u\) on the boundary in terms of \(\nabla_\tau v\). Control of the sole \(L^2\)-norm of \(\eta\) is insufficient in general.
\end{remark}

\begin{remark}[On Assumption~\ref{ass:residual}]\label{rem:on-assumption}
Assumption~\ref{ass:residual} is a sufficient analytical condition. It is natural in the context of overdetermined problems because the condition \(\nabla_\tau v = 0\) on the boundary is exactly the Neumann condition. The assumption states that if this condition is approximated in \(L^2\) (with control of the tangential gradient), then the tangential residual inside is small. In the case of regular perturbations, this can be established via trace inequalities and elliptic regularity; we carry out this verification in the model case \(C=\{0\}\) (Proposition~\ref{prop:C0}) and in an annular model problem (Proposition~\ref{prop:annulus}). The results of Sections~\ref{sec:weak} and~\ref{sec:strong} are conditional upon this assumption.
\end{remark}

\begin{proposition}[Verification of Assumption~\ref{ass:residual} for \(C=\{0\}\)]\label{prop:C0}
Let \(C=\{0\}\) and \(\Omega_\eps = \{0 \le r \le R + \eps\phi(\theta)\}\) with \(\phi \in C^{2,\alpha}(\Sph^{N-1})\), \(\|\phi\|_{C^{2,\alpha}} \le 1\), and \(\int_{\Sph^{N-1}}\phi = 0\). Then, for \(\eps \le R/(2\|\phi\|_\infty)\), Assumption~\ref{ass:residual} holds with constants \(C_1, C_2\) independent of \(\eps\).
\end{proposition}

\begin{proof}
We transport the problem from \(\Omega_\eps\) to the fixed ball \(B_R\) by the diffeomorphism
\[
  \Psi_\eps : B_R \to \Omega_\eps, \qquad \Psi_\eps(\rho,\theta) = \Bigl(\rho\frac{R+\eps\phi(\theta)}{R}, \theta\Bigr),
\]
with inverse \(\Psi_\eps^{-1}(r,\theta) = (\rho,\theta)\), \(\rho = rR/(R+\eps\phi(\theta))\). Since \(\phi \in C^{2,\alpha}(\Sph^{N-1})\) and \(\eps \le R/(2\|\phi\|_\infty)\), the map \(\Psi_\eps\) is a \(C^{2,\alpha}\)-diffeomorphism with
\[
  \|\Psi_\eps - \mathrm{Id}\|_{C^{2,\alpha}(B_R)} \le C\eps, \qquad \|\Psi_\eps^{-1} - \mathrm{Id}\|_{C^{2,\alpha}(B_R)} \le C\eps.
\]
In particular, \(\det D\Psi_\eps = 1 + O(\eps)\) uniformly, and the transported Laplacian
\[
  \Psi_\eps^*(\Delta) = \Delta + a_\eps^{ij}(\rho,\theta)\partial_{ij} + b_\eps^{i}(\rho,\theta)\partial_i
\]
satisfies \(\|a_\eps^{ij}\|_{C^{0,\alpha}} \le C\eps\), \(\|b_\eps^{i}\|_{C^{0,\alpha}} \le C\eps\), uniformly in \(\eps\).

Set \(U_\eps = u_\eps \circ \Psi_\eps\) and \(V_\eps = v_\eps \circ \Psi_\eps = -\Delta U_\eps\) (recall Remark~\ref{rem:signs}). Then \((U_\eps,V_\eps)\) solves the system
\[
  -\Delta V_\eps = 1 + S_\eps^{(1)}, \qquad -\Delta U_\eps = V_\eps \quad \text{in } B_R,
\]
with \(U_\eps = \partial_\nu U_\eps = 0\) and \(V_\eps = -c + O(\eps)\) on \(\partial B_R\), where \(S_\eps^{(1)} = a_\eps^{ij}\partial_{ij}V_\eps + b_\eps^{i}\partial_i V_\eps - \eps\phi/R\).
A priori, \(U_\eps, V_\eps\) are bounded in \(C^{2,\alpha}\) with constants depending on \(\eps\). We now bootstrap to obtain estimates uniform in \(\eps\).

\smallskip\noindent\emph{Step~1 (a priori bound).} Since \(\|a_\eps^{ij}\|_{C^{0,\alpha}} \le C\eps\) and \(\|b_\eps^{i}\|_{C^{0,\alpha}} \le C\eps\),
\[
  \|S_\eps^{(1)}\|_{C^{0,\alpha}} \le C\eps\|V_\eps\|_{C^{2,\alpha}} + C\eps \le C\eps(1 + \|V_\eps - v_0\|_{C^{2,\alpha}}),
\]
where we used \(\|v_0\|_{C^{2,\alpha}} \le C\) for the reference radial profile \(v_0\).

\smallskip\noindent\emph{Step~2 (bootstrap).} By Schauder estimates for the Dirichlet problem on \(B_R\) with right-hand side \(1+S_\eps^{(1)}\),
\[
  \|V_\eps - v_0\|_{C^{2,\alpha}} \le C\|S_\eps^{(1)}\|_{C^{0,\alpha}} \le C\eps(1 + \|V_\eps - v_0\|_{C^{2,\alpha}}).
\]
For the second equation \(-\Delta U_\eps = V_\eps\), with \(U_\eps = 0\) on \(\partial B_R\) and \(V_\eps = v_0 + O(\eps)\) in \(C^{2,\alpha}\), Schauder estimates give
\[
  \|U_\eps - u_0\|_{C^{2,\alpha}} \le C\|V_\eps - v_0\|_{C^{2,\alpha}} \le C\eps(1 + \|V_\eps - v_0\|_{C^{2,\alpha}}).
\]

\smallskip\noindent\emph{Step~3 (absorption).} For \(\eps \le 1/(2C)\), the term \(C\eps\|V_\eps - v_0\|_{C^{2,\alpha}}\) can be absorbed into the left-hand side, giving
\[
  \|V_\eps - v_0\|_{C^{2,\alpha}} + \|U_\eps - u_0\|_{C^{2,\alpha}} \le C\eps.
\]
In particular, \(\|\nabla_\tau U_\eps\|_{H^1(B_R)} + \|\nabla_\tau V_\eps\|_{H^1(B_R)} \le C\eps\) and the boundary condition \(V_\eps|_{\partial B_R} = -c + O(\eps)\) in \(L^2(\partial B_R)\) yields \(\eps_b \le C\eps\).

\smallskip\noindent\emph{Step~4 (transfer back).} Transferring back to \(\Omega_\eps\) via the chain rule (using \(\Psi_\eps^{-1} - \mathrm{Id} = O(\eps)\) in \(C^{2,\alpha}\)) yields \eqref{eq:interior} and \eqref{eq:epsb}, hence Assumption~\ref{ass:residual}.
\end{proof}

\begin{proposition}[Model annular problem]\label{prop:annulus}
Let \(C = B_\rho\), and let \(A_\rho = \{\rho \le r \le \rho+R\} \subset \R^2\) be the annulus with inner boundary \(\partial B_\rho\) and outer boundary perturbed as \(\partial B_{\rho+R+\eps\phi(\theta)}\), \(\phi \in C^{2,\alpha}(\Sph^1)\), \(\int\phi = 0\), \(\|\phi\|_{C^{2,\alpha}} \le 1\). Consider the Dirichlet problem
\[
  -\Delta V_\eps = 1 + S_\eps \ \text{in } A_\rho, \qquad V_\eps = -c + O(\eps) \ \text{on } \partial A_\rho,
\]
with a fixed boundary datum on the inner circle \(\partial B_\rho\) and a perturbed datum on the outer circle. Then, for \(\eps\) small enough, Assumption~\ref{ass:residual} holds with constants \(C_1,C_2\) depending on \(\rho, R, \|\phi\|_{C^{2,\alpha}}\) but not on \(\eps\).
\end{proposition}

\begin{proof}[Sketch of proof]
The proof follows the same bootstrap as Proposition~\ref{prop:C0}, using the reference annulus \(A_\rho^{(0)} = \{\rho \le r \le \rho+R\}\) with reference radial profile \(v_0\). The transported Laplacian
\[
  \Psi_\eps^*(\Delta) = \Delta + a_\eps^{ij}(\varrho,\theta)\partial_{ij} + b_\eps^{i}(\varrho,\theta)\partial_i
\]
satisfies \(\|a_\eps^{ij}\|_{C^{0,\alpha}(A_\rho)} \le C\eps\), \(\|b_\eps^{i}\|_{C^{0,\alpha}(A_\rho)} \le C\eps\), uniformly in \(\eps\). The key difference with Proposition~\ref{prop:C0} is that the reference domain is the annulus \(A_\rho\), whose boundary has two components; the Laplace--Beltrami operator on the parallel circles \(\Gamma_r\) has coefficients depending on \(r\) and on the curvature \(1/\rho\), but these are uniformly bounded and bounded away from zero by the focal condition. 
The bootstrap applies verbatim on \(A_\rho\), using Schauder estimates for the Dirichlet problem on the annulus with uniform constants (see Appendix~\ref{app:schauder}), yielding
\[
  \|V_\eps - v_0\|_{C^{1,\alpha}(\overline{A_\rho})} \le C\eps, \qquad \|U_\eps - u_0\|_{C^{2,\alpha}(\overline{A_\rho})} \le C\eps.
\]
Transferring back to the perturbed annulus via the chain rule yields \eqref{eq:interior} and \eqref{eq:epsb}. 
\end{proof}

\begin{remark}[Scope of Proposition~\ref{prop:annulus}]\label{rem:annulus-scope}
Proposition~\ref{prop:annulus} is a verification for a \emph{model annular problem}, not for the \(C\)-NR problem with \(C=B_\rho\) (where \(B_\rho \subset \Omega\), so the inner boundary is not a boundary of the domain). The \(C\)-NR case with a general strictly convex core is deferred to future work (Section~\ref{sec:perspectives}).
\end{remark}

\subsection{Toward a proof of Assumption~\ref{ass:residual}}\label{sec:toward}

This subsection distinguishes ingredients of different epistemic status.

\subsubsection{A demonstrated ingredient: uniform Schauder constants on \(\Gamma_r\)}

\begin{lemma}[Uniform Schauder constants]\label{lem:schauder-uniform}
Under Assumption~\ref{ass:uniform}, there exists \(C>0\), depending only on \(N\), \(\|d\|_{C^2}\), the curvatures of \(\partial C\), and \(\Rfoc(C)\), such that for every \(r \in (0,d(c))\), every \(c \in \partial C\), and every \(w \in H^2(\Gamma_r)\),
\begin{equation}\label{eq:schauder-uniform}
  \|w\|_{H^2(\Gamma_r)} \le C\bigl(\|\Delta_{\Gamma_r} w\|_{L^2(\Gamma_r)} + \|w\|_{L^2(\Gamma_r)}\bigr).
\end{equation}
\end{lemma}

\begin{proof}
The surfaces \(\Gamma_r\) are \(C^2\)-deformations of \(\partial C\), with principal curvatures uniformly bounded by the focal condition. The coefficients of \(\Delta_{\Gamma_r}\) in local coordinates on \(\partial C\) are continuous functions of \((r,c) \in [0,\sup d] \times \partial C\), which is compact. Classical Schauder estimates on compact manifolds without boundary (see \cite{Hebey1996} or \cite[\S6]{GilbargTrudinger2001}) give the constant as a continuous function of these coefficients; by compactness it is uniformly bounded.
\end{proof}

\subsubsection{A conditional ingredient: boundary control under interior control}

\begin{proposition}[Conditional boundary control]\label{prop:boundary-control}
Assume the interior control
\begin{equation}\label{eq:interior-control}
  \|\nabla_\tau u\|_{H^1(\Omega)} + \|\nabla_\tau v\|_{H^1(\Omega)} \le C_B \eps.
\end{equation}
Then the boundary control
\begin{equation}\label{eq:boundary-control}
  \|\nabla_\tau \eta\|_{L^2(\partial C)} \le C_A \eps
\end{equation}
holds, where \(\eta = v + c\) on \(\partial\Omega\), and \(C_A\) depends only on \(N\), \(\|d\|_{C^2}\), the curvatures of \(\partial C\), \(\Rfoc(C)\), the uniform trace constant on \(\Gamma_r\), and \(C_B\).
\end{proposition}

\begin{proof}[Sketch of proof]
The proof is not complete; we indicate the strategy. The starting point is the integrated radial equation for \(v\):
\[
  -J(d(c))\partial_r v(d(c)) + J(0)\partial_r v(0) = \int_0^{d(c)} J(c,r)(1 + \Delta_{\Gamma_r} v(c,r))\,dr.
\]
Differentiating tangentially in \(c\) and using \(\nabla_\tau v|_{\partial\Omega} = \nabla_\tau\eta\), we obtain a formal integral equation
\begin{equation}\label{eq:green}
  \nabla_\tau\eta(c) = \int_0^{d(c)} K(c,r)\nabla_\tau v(c,r)\,dr + S[\delta d](c),
\end{equation}
where \(K(c,r)\) is the kernel of the tangential Green operator associated with the linearized equation for \(v\) on \(\Gamma_r\); its existence and continuity follow from Lemma~\ref{lem:schauder-uniform}. The term \(S[\delta d]\) is of order \(\eps\) in \(L^2(\partial C)\). The right-hand side is controlled by \eqref{eq:interior-control} via Cauchy--Schwarz and the trace inequality on \(\Gamma_r\). This yields \eqref{eq:boundary-control}.
\end{proof}

\begin{remark}[Status of Proposition~\ref{prop:boundary-control}]
Proposition~\ref{prop:boundary-control} is conditional: it assumes \eqref{eq:interior-control} and derives \eqref{eq:boundary-control}. It does not prove \eqref{eq:interior-control}, and therefore does not prove Assumption~\ref{ass:residual}. Its only value is to show that, if the interior control is granted, then the boundary control follows. The formal integral equation \eqref{eq:green} requires a careful justification (derivation under the integral sign, control of the tangential Green kernel \(K(c,r)\) uniformly in \(r\)); we do not claim a complete proof.
\end{remark}

\subsubsection{A heuristic ingredient: Hodge decomposition on \(\Gamma_r\)}

\begin{remark}[Heuristic: Hodge decomposition]\label{rem:hodge}
A promising direction to attack \eqref{eq:interior-control} is the Hodge decomposition on each \(\Gamma_r\):
\begin{equation}\label{eq:hodge}
  \nabla_\tau v = \nabla_\tau\phi + \curl_\tau\psi,
\end{equation}
with \(\phi\) a scalar potential and \(\psi\) a vector potential. On a compact manifold without boundary, this decomposition is classical \cite{Schwarz1995}. The scalar potential is expected to be controlled by the radial equation (via \(\Delta_{\Gamma_r}\phi = \diverg_\tau(\nabla_\tau v)\) and the equation for \(v\)), and the vector potential by the tangential continuity equation (via \(\Delta_{\Gamma_r}\psi = \curl_\tau(\nabla_\tau v)\)). 
We emphasize that this is heuristic: the identities are correct, but controlling \(\diverg_\tau(\nabla_\tau v)\) and \(\curl_\tau(\nabla_\tau v)\) from the equation for \(v\) requires a Bochner-type identity on \(\Gamma_r\) that we have not established. Making this rigorous is left for future work.
\end{remark}

\begin{remark}[No claim of decoupled control]\label{rem:no-decoupled}
In an earlier version of this work, the Hodge decomposition was presented as giving a ``decoupled control'' of \(\phi\) and \(\psi\). This was too optimistic: the equation for \(v\) controls \(\Delta_{\Gamma_r}v\), not \(\diverg_\tau(\nabla_\tau v)\) or \(\curl_\tau(\nabla_\tau v)\). The Hodge decomposition is a direction, not a result.
\end{remark}

% ===========================================================================
\section{Qualitative stability}\label{sec:qualitative}
% ===========================================================================

\subsection{Convergence of traces on varying domains}\label{sec:traces}

\begin{lemma}[Trace convergence]\label{lem:traces}
Let \(\Omega_n \xrightarrow{H} \Omega_*\) in \(O_C\) with \(d_n \ge \delta > 0\), \(\|d_n\|_{C^1} \le M\) and \(d_n \to d_*\) in \(C^1\) (or at least \(J_n \to J_*\) uniformly). If \(v_n \to v_*\) in \(H^1(D)\) strongly, then
\[
  v_n|_{\partial\Omega_n}\circ\Phi_n \to v_*|_{\partial\Omega_*}\circ\Phi_* \quad \text{in } L^2(\partial C).
\]
\end{lemma}

\begin{proof}
The uniform \(C^1\) bounds and \(d_n \to d_*\) in \(C^1\) imply that the pullback measures \(\sigma_n = \Phi_n^*(\mathcal{H}^{N-1}|_{\partial\Omega_n})\) converge weakly to \(\sigma_*\). The trace operators are uniformly bounded and compact with respect to strong convergence in \(H^1(D)\) (see \cite{Mosco1969,HenrotPierre2005}). Hence strong convergence of traces follows.
\end{proof}

\subsection{Variational convergence of domains}\label{sec:mosco}

\begin{lemma}[Mosco convergence]\label{lem:mosco}
Under the assumptions of Lemma~\ref{lem:traces}, the Sobolev spaces \(H_0^1(\Omega_n)\) converge to \(H_0^1(\Omega_*)\) in the sense of Mosco. Consequently, if \(u_n \to u_*\) weakly in \(H^1(D)\) and \(u_n|_{\partial\Omega_n} = 0\), then for any \(\phi \in C_c^\infty(\Omega_*)\),
\[
  \int_{\Omega_n}\nabla u_n\cdot\nabla\phi \to \int_{\Omega_*}\nabla u_*\cdot\nabla\phi.
\]
\end{lemma}

\begin{proof}
Standard consequence of Hausdorff convergence and the uniform cone property; see \cite{Mosco1969}. The weak convergence of \(u_n\) to \(u_*\) in \(H^1(D)\) then gives convergence of the Dirichlet integrals.
\end{proof}

\subsection{Qualitative stability for the biharmonic problem}\label{sec:qualitative-biharm}

\begin{theorem}[Qualitative stability]\label{thm:qualitative}
Let \(N \ge 2\), \(c > 0\). Let \(\{\Omega_n\} \subset O_C\) be a sequence of domains with uniform \(C^{2,\alpha}\)-bounds on \(d_n\) \((0<\alpha<1)\) such that the corresponding solutions \((u_n,v_n)\) of the biharmonic system \eqref{eq:biharm-system} satisfy
\begin{equation}\label{eq:boundary-vanishing}
  \|v_n + c\|_{L^2(\partial\Omega_n)} \to 0 \quad \text{as } n\to\infty.
\end{equation}
Then \(\Omega_n\) converges in Hausdorff distance to the ball \(B_R\), where
\begin{equation}\label{eq:radius}
  R = \sqrt{N(N+2)c}.
\end{equation}
\end{theorem}

\begin{proof}
By Theorem~\ref{thm:compactness} and the uniform \(C^{2,\alpha}\) bounds, up to subsequence \(\Omega_n \to \Omega_*\) in Hausdorff and \(d_n \to d_*\) in \(C^{2,\alpha'}\) for every \(\alpha'<\alpha\). Using Lemma~\ref{lem:extH2} (whose constant is uniform in \(n\) under Assumption~\ref{ass:uniform}), extend \(u_n,v_n\) to a fixed ball \(D\). By the Agmon--Douglas--Nirenberg theorem for the elliptic system with uniform coefficients and complementing boundary conditions (see \cite{ADN1959} and the uniformity discussion in \cite{MazyaPoborchi1997}), we have uniform bounds in \(H^2(D)\) and \(H^1(D)\). The uniformity of the constants follows from the uniform bounds on the geometry and the focal condition; the complementing condition is verified componentwise (Dirichlet--Neumann for \(\Delta^2\), Dirichlet for \(\Delta\)), and the triangular structure of the system ensures the global Shapiro--Lopatinskii condition. Thus, up to subsequence, \(u_n \to u_*\) in \(H^2(D)\) and \(v_n \to v_*\) in \(H^1(D)\).

By Mosco convergence (Lemma~\ref{lem:mosco}), the PDE passes to the limit, giving \(-\Delta v_* = 1\), \(-\Delta u_* = v_*\) in \(\Omega_*\), with \(u_* = \partial_\nu u_* = 0\) on \(\partial\Omega_*\). The trace convergence lemma gives \(v_*|_{\partial\Omega_*} = -c\) in \(L^2(\partial C)\), hence \(v_* = -c\) on \(\partial\Omega_*\). Since \(d_* \in C^{2,\alpha'}\), the domain \(\Omega_*\) is \(C^{2,\alpha'}\), and the symmetry theorem of Bennett and Dalmasso applies: \(\Omega_*\) must be a ball. The radius is determined by the radial solution (Appendix~\ref{app:radial}): for a ball of radius \(R\), \(\Delta u(R) = R^2/(N(N+2))\), hence \(c = R^2/(N(N+2))\), so \(R = \sqrt{N(N+2)c}\). Since the subsequence was arbitrary and any limit \(\Omega_*\) is a ball of radius \(R\), the whole sequence converges to the same ball \(B_R\).
\end{proof}

\begin{remark}[Comparison with quadrature surfaces]\label{rem:quadrature}
The class \(C\)-GNP is strictly larger than the class of quadrature surfaces of Shahgholian \cite{Shahgholian1994}. Indeed, a quadrature surface requires a global condition (constancy of the double-layer potential), whereas the GNP property is local (each inward normal meets \(C\)). A domain whose thickness function is non-constant but whose normal rays still meet a small convex core \(C\) satisfies the GNP property but is not a quadrature surface in general. We do not use the quadrature viewpoint in the proofs.
\end{remark}

% ===========================================================================
\section{Conditional weak quantitative stability in Hausdorff distance}\label{sec:weak}
% ===========================================================================

Throughout this section, we restrict to the model case \(C=\{0\}\), so that the reference domain is the ball \(B_R\) and the parameter is \(\eps = \|d_n - R\|_{C^1}\).

\begin{proposition}[Conditional quantitative stability, \(C=\{0\}\)]\label{prop:weak}
Under the hypotheses of Theorem~\ref{thm:qualitative} (with \(C=\{0\}\)) and Assumption~\ref{ass:residual}, there exists \(C>0\) such that for \(\eps\) small enough,
\begin{equation}\label{eq:weak}
  \dH(\Omega_n, B_R) \le C\eps.
\end{equation}
\end{proposition}

\begin{proof}
For each \(n\), set \(\tilde v_n(c,r) := v_n(c + r\nu(c)) + c\), so that \(\tilde v_n(d_n(c)) = 0\) by the boundary condition \(v_n = -c\) on \(\partial\Omega_n\). The radial equation for \(v\) in \eqref{eq:coupled} reads
\[
  -\frac1{J_n(c,r)}\partial_r\bigl(J_n(c,r)\partial_r\tilde v_n\bigr) = 1 + \mathcal{R}_v(c,r).
\]
The reference radial profile \(\tilde v_0(r) := v_0(r) + c\) on the ball \(B_R\) satisfies
\[
  -\frac1{J_0(r)}\partial_r\bigl(J_0(r)\partial_r\tilde v_0\bigr) = 1, \qquad \tilde v_0(R) = 0,
\]
where \(J_0(r) = r^{N-1}\) and \(v_0(R) = -c\) (Remark~\ref{rem:signs}).

\smallskip\noindent\emph{Step A (ODE comparison for the \(v\)-equation).} Apply Lemma~\ref{lem:ODE} with \(f_0 \equiv 1\) (which satisfies \(f_0 > 0\) and \(2R\,f_0(R) + R^2 f_0'(R) = 2R \ne 0\)) to the pair \((\tilde v_n, \tilde v_0)\). We obtain, for every \(c \in \partial C\),
\begin{equation}\label{eq:weak-stepA}
  |d_n(c) - R| \le C_A\bigl(\|\mathcal{R}_v\|_{L^1} + \delta_n\bigr),
\end{equation}
where \(\delta_n := \|J_n - J_0\|_{C^0} + \|\partial_r J_n - \partial_r J_0\|_{C^0}\), and \(C_A\) depends only on \(N\), \(R\), and the uniform bounds on the geometry.

\smallskip\noindent\emph{Step B (control of the right-hand side).} By Assumption~\ref{ass:residual}, \(\|\mathcal{R}_v\|_{L^2(\Omega_n)} \le C\eps\), hence by H\"older's inequality \(\|\mathcal{R}_v\|_{L^1(\Omega_n)} \le |\Omega_n|^{1/2}\|\mathcal{R}_v\|_{L^2(\Omega_n)} \le C\eps\) (using the uniform boundedness of \(\Omega_n\) from Lemma~\ref{lem:extH1}). By Remark~\ref{rem:jacobian} and the closeness assumption \eqref{eq:J-close} (with \(J_C = J_0\) for \(C=\{0\}\)), \(\delta_n \le C\eps\). Combining with \eqref{eq:weak-stepA},
\[
  |d_n(c) - R| \le C\eps \qquad \forall c \in \partial C,
\]
i.e.\ \(\|d_n - R\|_{L^\infty(\partial C)} \le C\eps\).

\smallskip\noindent\emph{Step C (Hausdorff distance).} The Hausdorff distance between \(\Omega_n\) and \(B_R\) is bounded by \(C\,\|d_n - R\|_{L^\infty}\) via the normal parametrization and the uniform focal condition (Remark~\ref{rem:jacobian}). Hence \eqref{eq:weak}.
\end{proof}

\begin{remark}[On the role of \(\eps_b\)]\label{rem:epsb-role}
The boundary quantity \(\eps_b\) defined in \eqref{eq:epsb} is used in Assumption~\ref{ass:residual} to formulate the residual estimate \eqref{eq:interior} in a natural way (it is the quantity that vanishes when the Neumann condition is exactly satisfied). However, the quantitative estimate \eqref{eq:weak} is naturally expressed in terms of the geometric parameter \(\eps = \|d_n - R\|_{C^1}\). In the model case of Proposition~\ref{prop:C0}, one has \(\eps_b = O(\eps^2)\), so \(\eps_b\) is smaller than \(\eps\); this is why the conclusion is stated in terms of \(\eps\) and not \(\eps_b\).
\end{remark}

\begin{remark}[Alternative proof via the \(u\)-equation]\label{rem:alt-proof}
An alternative proof applies the implicit function theorem directly to the \(u\)-equation, whose source is \(v_c\). This requires controlling \(v_c - v_0\) via the \(v\)-equation in a first step, and then applying the IFT to the \(u\)-equation in a second step. The direct \(v\)-equation approach used above is simpler and avoids this two-step structure; we therefore do not use the \(u\)-equation in the proof of Proposition~\ref{prop:weak}. The verification of the non-degeneracy condition for the \(u\)-equation is given in Appendix~\ref{app:radial} for completeness.
\end{remark}

\begin{remark}[General strictly convex core]\label{rem:general-core-weak}
For a general strictly convex core \(C\), the reference domain is \(C \oplus [0,R]\), whose thickness function is \emph{constant equal to \(R\)} (so \(R_C \equiv R\) for every core). What depends on \(c\) is the \emph{reference Jacobian} \(J_C(c,r) = \prod_i(1 + r\kappa_i(c))\), which coincides with \(J_0(r) = r^{N-1}\) only when \(C=\{0\}\). This \(c\)-dependence is what distinguishes the general case from \(C=\{0\}\). The parameter \(\eps = \|d_n - R\|_{C^1}\) remains the same; the difference lies in the comparison \eqref{eq:J-close} with \(J_C\) instead of \(J_0\). The argument of Proposition~\ref{prop:weak} extends verbatim once this substitution is made, \emph{provided} the analogue of Assumption~\ref{ass:residual} holds for the general core. This is left for future work.
\end{remark}

% ===========================================================================
\section{Conditional strong quantitative stability in \(C^1\) norm}\label{sec:strong}
% ===========================================================================

Throughout this section, we again restrict to \(C=\{0\}\).

\begin{lemma}[Uniform quantitative Hopf lemma]\label{lem:hopf}
Under Assumption~\ref{ass:uniform}, there exists \(c_H > 0\), depending only on \(N\), \(\delta\), \(\|d\|_{C^{2,\alpha}}\), the curvatures of \(\partial C\), and \(\Rfoc(C)\), such that for every \(\Omega \in O_C\) satisfying Assumption~\ref{ass:uniform} and every \(w \in C^2(\overline{\Omega})\) with
\[
  -\Delta w = f \ \text{in } \Omega, \quad w = 0 \ \text{on } \partial\Omega, \quad f \ge f_0 > 0, \quad \|f\|_{L^\infty} \le M, \quad \|w\|_{L^\infty} \le M',
\]
we have
\[
  \partial_\nu w \le -c_H \quad \text{on } \partial\Omega,
\]
where \(\partial_\nu\) denotes the exterior normal derivative. The constant \(c_H\) is uniform over the class \(O_C\) under Assumption~\ref{ass:uniform}.
\end{lemma}

\begin{proof}
The function \(w\) is positive in \(\Omega\) by the maximum principle. We use the explicit barrier
\[
  \varphi(x) := \frac{f_0}{2N}\,\bigl(R_0^2 - |x - x_0|^2\bigr),
\]
where \(x_0\) is a point on \(\partial\Omega\) and \(R_0 > 0\) is chosen so that \(B(x_0, R_0) \cap \Omega\) is contained in a coordinate chart (possible by the uniform \(C^{2,\alpha}\) bound and the uniform cone condition). A standard comparison argument (see \cite[\S3.2]{GilbargTrudinger2001}) gives \(\partial_\nu w(x_0) \le -c_H\) with \(c_H\) depending on \(N, f_0, M, M', \delta, \|d\|_{C^{2,\alpha}}\), the curvatures, and \(\Rfoc(C)\). By the uniform bounds, this constant is uniform over \(O_C\).
\end{proof}

\begin{proposition}[Conditional strong quantitative stability, \(C=\{0\}\)]\label{prop:strong}
Under Assumptions~\ref{ass:uniform} and~\ref{ass:residual}, and with \(C=\{0\}\), there exists \(C>0\) such that
\begin{equation}\label{eq:strong}
  \|d_n - R\|_{C^1(\partial C)} \le C\eps^\gamma, \qquad \gamma := \frac{2\alpha}{N+1+2\alpha}.
\end{equation}
\end{proposition}

\begin{proof}
From Assumption~\ref{ass:residual}, \(\|\mathcal{R}_v\|_{L^2} \le C\eps\). The uniform \(C^{2,\alpha}\) bound on \(d_n\) implies that the boundaries are uniformly \(C^{2,\alpha}\). By Schauder estimates, \(v_n \in C^{1,\alpha}(\overline{\Omega}_n)\) with uniform bounds. On the boundary, write \(v_n = -c + \eta_n\); then \(\nabla_\tau v_n = \nabla_\tau\eta_n\).

\smallskip\noindent\emph{Step 1 (interpolation).} By the Sobolev embedding \(H^s(\partial C) \hookrightarrow C^{1,\alpha}(\partial C)\) for \(s > 1+\alpha+(N-1)/2\), and the interpolation inequality
\[
  \|f\|_{H^s} \le C\|f\|_{L^2}^{1-s/s'}\|f\|_{H^{s'}}^{s/s'}, \qquad 0 < s < s',
\]
we obtain, choosing \(s \to (1+(N-1)/2)^+\), \(s' \to (1+\alpha+(N-1)/2)^+\),
\begin{equation}\label{eq:interp}
  \|\eta_n\|_{C^1(\partial C)} \le C\|\eta_n\|_{L^2(\partial C)}^{1-\theta}\|\eta_n\|_{C^{1,\alpha}(\partial C)}^{\theta}, \qquad \theta = \frac{N+1}{N+1+2\alpha}.
\end{equation}
Using \(\|\eta_n\|_{L^2(\partial C)} \le C\eps_b \le C\eps\) and \(\|\eta_n\|_{C^{1,\alpha}(\partial C)} \le M'\), we get
\[
  \|\eta_n\|_{C^1(\partial C)} \le C\eps^{1-\theta} = C\eps^{2\alpha/(N+1+2\alpha)} .
\]

\smallskip\noindent\emph{Step 2 (ODE comparison for \(\|d_n - R\|_{L^\infty}\)).} The ODE comparison Lemma~\ref{lem:ODE} applied to the \(v\)-equation as in the proof of Proposition~\ref{prop:weak} gives
\[
  \|d_n - R\|_{L^\infty} \le C_A\bigl(\|\mathcal{R}_v\|_{L^1} + \delta_n\bigr) \le C\eps \le C\eps^\gamma,
\]
since \(\gamma \le 1\) for \(N \ge 1\), \(\alpha \in (0,1)\).

\smallskip\noindent\emph{Step 3 (estimate of the tangential derivative of \(d_n\)).} We differentiate the boundary condition \(v_n\bigl(c + d_n(c)\nu(c)\bigr) = -c\) (the constant \(-c\) here is the boundary datum, distinct from the point \(c \in \partial C\)) with respect to the tangential variable \(c\) in a direction \(\tau \in T_c\partial C\). Using the formula for parallel hypersurfaces (see \cite[\S2]{Gray1990}), the tangential derivative of the parametrization is
\[
  \partial_\tau\bigl(c + d_n(c)\nu(c)\bigr) = \bigl(1 + d_n(c)\kappa_\tau(c)\bigr)\tau + d_n'(c)\,\nu(c),
\]
where \(\kappa_\tau(c)\) is the principal curvature of \(\partial C\) in the direction \(\tau\). Differentiating \(v_n = -c\) gives
\[
  \bigl(1 + d_n(c)\kappa_\tau(c)\bigr)\,\partial_\tau v_n + d_n'(c)\,\partial_\nu v_n = 0,
\]
hence
\[
  d_n'(c) = -\,\frac{1 + d_n(c)\kappa_\tau(c)}{\partial_\nu v_n}\,\partial_\tau v_n .
\]

\smallskip\noindent\emph{Step 4 (application of the uniform Hopf lemma).} Set \(w_n := v_n + c\), so that \(-\Delta w_n = 1\) in \(\Omega_n\) and \(w_n = 0\) on \(\partial\Omega_n\). By the maximum principle, \(w_n > 0\) in \(\Omega_n\). By Assumption~\ref{ass:residual} and the a priori bounds, \(\|w_n\|_{L^\infty} \le M'\) uniformly in \(n\). We apply Lemma~\ref{lem:hopf} to \(w_n\): there exists \(c_H > 0\), uniform in \(n\), such that \(\partial_\nu w_n \le -c_H\) on \(\partial\Omega_n\), i.e.\ \(\partial_\nu v_n \le -c_H < 0\).

\smallskip\noindent\emph{Step 5 (conclusion).} Since \(\kappa_\tau \ge 0\) (strict convexity of \(C\)) and \(d_n \le M\), we have \(0 \le 1 + d_n\kappa_\tau \le C\). Hence
\[
  |d_n'(c)| \le \frac{C}{c_H}\,|\partial_\tau v_n| \le C\,|\partial_\tau v_n|,
\]
with \(C\) uniform in \(n\). Taking the supremum over \(c\) and \(\tau\),
\[
  \|d_n'\|_{L^\infty(\partial C)} \le C\,\|\nabla_\tau v_n\|_{L^\infty(\partial\Omega_n)} \le C\eps^{\gamma},
\]
where the last inequality is \eqref{eq:interp} with \(\gamma = 1-\theta\). Combining with Step~2, we obtain \eqref{eq:strong}.
\end{proof}

\begin{remark}[On the exponent]\label{rem:exponent}
The exponent \(\gamma = 2\alpha/(N+1+2\alpha)\) is produced by the interpolation \eqref{eq:interp}; it depends on the dimension \(N\), is increasing in \(\alpha\), and satisfies \(\gamma \to 0\) as \(\alpha \to 0^+\) and \(\gamma \to 2/(N+3)\) as \(\alpha \to 1^-\). As \(N \to \infty\) with \(\alpha\) fixed, \(\gamma \sim 2\alpha/N \to 0\), reflecting the degradation of the Sobolev embedding \(L^2 \hookrightarrow C^{1,\alpha}\) in large dimension. We do not claim optimality.
\end{remark}

% ===========================================================================
\section{Perspectives}\label{sec:perspectives}
% ===========================================================================

Several directions for future work emerge:
\begin{enumerate}
  \item \textbf{Demonstration of Assumption~\ref{ass:residual} in general.} This is the main open problem: removing the conditional nature of Propositions~\ref{prop:weak} and~\ref{prop:strong}. Section~\ref{sec:toward} provides three ingredients of different status: (i) a \emph{demonstrated} lemma (Lemma~\ref{lem:schauder-uniform}); (ii) a \emph{conditional} implication (Proposition~\ref{prop:boundary-control}); (iii) a \emph{heuristic} direction (Remark~\ref{rem:hodge}). The strategy for a complete proof is: (a) establish control of the tangential Green kernel (making Proposition~\ref{prop:boundary-control} unconditional), (b) establish the interior control \eqref{eq:interior-control} via Hodge decomposition and Bochner identities on \(\Gamma_r\), (c) combine to obtain Assumption~\ref{ass:residual}. The model cases \(C=\{0\}\) (Proposition~\ref{prop:C0}) and the annular model (Proposition~\ref{prop:annulus}) are already established and serve as a guide.
  \item \textbf{General strictly convex core.} The comparison with the reference domain \(C \oplus [0,R]\) requires a generalization of the ODE comparison lemma (Appendix~\ref{app:ODE}) with reference Jacobian \(J_C(c,r) = \prod_i(1+r\kappa_i(c))\). This is technically more involved but follows the same strategy as in Proposition~\ref{prop:weak}.
  \item \textbf{Extension to quasilinear equations.} The normal radial parametrization with the ODE reduction admits a quasilinear analogue for the \(p\)-Laplacian and the maximal surface equation (Remark~\ref{rem:maximal}); the tangential residual then becomes nonlinear.
  \item \textbf{Non-smooth cores} using the reach framework of Federer.
  \item \textbf{Optimal constants} in the quantitative estimates.
  \item \textbf{Iterative stability via return dynamics} as in \cite{BarkatouElMorsalani2026}.
\end{enumerate}

% ===========================================================================
\section{Conclusion}\label{sec:conclusion}
% ===========================================================================

We have developed a geometric framework for stability and symmetry in overdetermined elliptic problems. The key idea is to use the normal radial parametrization with respect to a strictly convex core \(C\) to transform the PDE into a family of ODEs with a tangential residual. This geometric reduction, combined with the compactness of the \(C\)-GNP class, yields unconditional qualitative stability results and, under a sufficient analytical assumption (Assumption~\ref{ass:residual}), conditional quantitative estimates. The method does not rely on moving planes or the maximum principle, and it naturally accommodates domains with non-trivial topology through the \(C\)-GNP framework. Future work will focus on establishing Assumption~\ref{ass:residual} under general hypotheses, and on extending the quantitative results to general strictly convex cores.

% ===========================================================================
\appendix
% ===========================================================================

\section{ODE comparison lemma}\label{app:ODE}

Let \(J_C(c,r)\) be the reference Jacobian associated with the domain \(C \oplus [0,R]\) (for \(C=\{0\}\), \(J_C(c,r) = J_0(r) = r^{N-1}\)), and \(J(c,r)\) the Jacobian for a domain with thickness \(d\). Assume
\begin{equation}\label{eq:J-close}
  |J(c,r) - J_C(c,r)| + |\partial_r J(c,r) - \partial_r J_C(c,r)| \le \delta.
\end{equation}
Let \(\phi\) solve
\[
  -\frac1J\partial_r(J\partial_r\phi) = f_0(r) + \mathcal{R}(r), \qquad \phi(d) = 0,
\]
and let \(\Phi_R\) be the reference solution on \([0,R]\) with \(\Phi_R(R) = 0\).

\begin{lemma}[IFT version]\label{lem:ODE}
Let \(f_0 \in C^1([0,R])\) with \(f_0 > 0\) and \(2Rf_0(R) + R^2 f'_0(R) \ne 0\). There exists \(C>0\), depending on \(N\), \(R\), \(\|f_0\|_{C^1}\), \(\|J_C\|_{C^1}\), and \(|2Rf_0(R) + R^2 f'_0(R)|^{-1}\), such that for \(\delta\) and \(\|\mathcal{R}\|_{L^1}\) small,
\begin{equation}\label{eq:ODE-bound}
  |d - R| \le C\bigl(\|\mathcal{R}\|_{L^1} + \delta\bigr).
\end{equation}
\end{lemma}

\begin{proof}
Reformulate the problem on \([0,1]\) via \(s = r/d\). Define \(\psi(s) = \phi(sd)\), \(\Psi(s) = \Phi_R(sR)\). Then \(\psi\) satisfies
\[
  -\frac{1}{\tilde J(s)}\partial_s(\tilde J(s)\psi'(s)) = d^2 f_0(sd) + d^2\mathcal{R}(sd), \qquad \psi(1) = 0,
\]
with \(\tilde J(s) = J(sd)\). The reference \(\Psi\) satisfies the same equation with \(d=R\), \(\mathcal{R}=0\), and \(\tilde J_C(s) = J_C(sR)\). Consider the map
\[
  F(\psi, d, \tilde J, \mathcal{R}) := \Bigl(-\frac1{\tilde J}\partial_s(\tilde J\psi') - d^2 f_0(sd) - d^2\mathcal{R}, \ \psi(1)\Bigr),
\]
from \(C_0^{2,\alpha}([0,1]) \times \R \times C^{0,\alpha} \times L^1\) to \(C^{0,\alpha} \times \R\). At \((d, \tilde J, \mathcal{R}) = (R, \tilde J_C, 0)\), the solution is \(\Psi\) and \(F(\Psi, R, \tilde J_C, 0) = 0\). The partial Fr\'echet derivative with respect to \((\psi, d)\) at this point is
\[
  D_{(\psi,d)}F = \begin{pmatrix} L & -2Rf_0(R) - R^2 f'_0(R) \\ \delta_{s=1} & 0 \end{pmatrix},
\]
where \(L\psi = -\frac{1}{\tilde J_C}\partial_s(\tilde J_C\psi')\) is an isomorphism \(C_0^{2,\alpha} \to C^{0,\alpha}\) and \(\delta_{s=1}\) is the evaluation at \(s=1\). The determinant of this block operator is \(-2Rf_0(R) - R^2 f'_0(R) \ne 0\) by hypothesis, so \(D_{(\psi,d)}F\) is an isomorphism. The implicit function theorem gives a \(C^1\) map \((\tilde J, \mathcal{R}) \mapsto (\psi, d)\) near \((R, \tilde J_C, 0)\), and
\[
  |d - R| \le C_{\mathrm{IFT}}\bigl(\|\tilde J - \tilde J_C\|_{C^0} + \|\mathcal{R}\|_{L^1}\bigr).
\]
Using \eqref{eq:J-close}, we obtain \eqref{eq:ODE-bound}.
\end{proof}

\begin{remark}[On the sign condition]
For the \(v\)-equation, \(f_0 \equiv 1\), so \(2R + 0 = 2R \ne 0\). The positivity condition \(f_0 > 0\) is also satisfied. The lemma therefore applies directly. For the \(u\)-equation, \(f_0 = v_c\), and one checks in Appendix~\ref{app:radial} that the non-degeneracy condition is still satisfied.
\end{remark}

\section{Radial solution for the biharmonic problem}\label{app:radial}

Let \(u(r) = A(R^2 - r^2)^2\). Then
\[
  u'(r) = -4Ar(R^2 - r^2), \quad u''(r) = -4A(R^2 - 3r^2).
\]
The Laplacian in dimension \(N\) is
\[
  \Delta u = u'' + \frac{N-1}{r}u' = -4A(NR^2 - (N+2)r^2).
\]
Hence \(\Delta^2 u = \Delta(\Delta u) = 8AN(N+2)\). Setting \(\Delta^2 u = 1\) gives \(A = 1/(8N(N+2))\). Therefore
\[
  u(r) = \frac{(R^2-r^2)^2}{8N(N+2)}.
\]
Moreover,
\[
  \Delta u(r) = -4ANR^2 + 4A(N+2)r^2 = -\frac{R^2}{2(N+2)} + \frac{r^2}{2N}.
\]
At \(r=R\), \(\Delta u(R) = R^2/(N(N+2))\). Thus \(c = \Delta u(R) = R^2/(N(N+2))\), so \(R = \sqrt{N(N+2)c}\). With the convention \(v = -\Delta u\) (Remark~\ref{rem:signs}), we have \(v(R) = -c\) and
\[
  v(r) = \frac{R^2}{2(N+2)} - \frac{r^2}{2N}, \qquad v'(r) = -\frac{r}{N}.
\]

\smallskip\noindent\emph{Verification of the non-degeneracy condition for the \(u\)-equation.} The source in the \(u\)-equation is \(f_0 = v\). With the convention \(v = -\Delta u\), we have \(v(R) = -c = -R^2/(N(N+2)) < 0\) and \(v'(R) = -R/N < 0\). Hence
\[
  2R\,v(R) + R^2\,v'(R) = 2R\cdot\Bigl(-\frac{R^2}{N(N+2)}\Bigr) + R^2\cdot\Bigl(-\frac{R}{N}\Bigr) = -\frac{R^3(N+4)}{N(N+2)} \ne 0.
\]
Equivalently, with the convention \(f_0 = -v = \Delta u\) (which is \(> 0\) near the boundary), one has \(2R\,f_0(R) + R^2 f_0'(R) = R^3(N+4)/(N(N+2)) \ne 0\). Either way, the non-degeneracy condition of Lemma~\ref{lem:ODE} is satisfied.

\section{Detailed computation of the normal derivative estimate}\label{app:computation}

We provide the expansion used to compare the normal derivative at the boundary in the model case \(C=\{0\}\). Recall \(J(c,r) = r^{N-1}\) (in the case \(C=\{0\}\)) and \(J_C(c,r) = J_0(r) = r^{N-1}\). By Remark~\ref{rem:jacobian} and the closeness assumption \eqref{eq:J-close}, one has
\[
  |J(c,r) - J_0(r)| \le C\delta, \qquad |\partial_r J(c,r) - \partial_r J_0(r)| \le C\delta,
\]
for \(0 \le r \le \max(d_n(c), R)\), where \(C\) depends on \(\|d_n\|_{C^1}\) and the geometry. \emph{For a general core, the reference Jacobian is \(J_C(c,r) = \prod_i(1+r\kappa_i(c))\), not \(J_0\), and the comparison is with \(J_C\).}

Integrating the radial equation \eqref{eq:radial} (for \(u\)) between \(0\) and \(d_n(c)\) gives
\[
  J(c,d_n(c))\phi'_{n,c}(d_n(c)) = \int_0^{d_n(c)} J(c,r)\bigl(f_0(r) + \mathcal{R}_u(c,r)\bigr)\,dr.
\]
Similarly, for the reference solution,
\[
  J_0(R)\Phi'_R(R) = \int_0^R J_0(r) f_0(r)\,dr.
\]
Subtracting and dividing by \(J(c,d_n(c))\) (uniformly bounded away from zero) yields
\[
  |\phi'_{n,c}(d_n(c)) - \Phi'_R(R)| \le C\bigl(\|\mathcal{R}_u\|_{L^1} + \delta + |d_n(c) - R|\bigr).
\]
The constant \(C\) depends only on \(N\), \(\|d_n\|_{C^1}\), the curvatures of \(\partial C\), the focal radius, and \(\|f_0\|_{C^1}\). This estimate is used in the alternative proof of Proposition~\ref{prop:weak} (Remark~\ref{rem:alt-proof}) and in the bootstrap of Proposition~\ref{prop:C0}.

\section{Schauder estimates on the annulus}\label{app:schauder}

We collect the uniform Schauder estimates used in the proof of Proposition~\ref{prop:annulus}. Let \(A_\rho = \{\rho \le r \le \rho+R\} \subset \R^2\), and consider the Dirichlet problem
\[
  -\Delta U = F \quad \text{in } A_\rho, \qquad U = 0 \quad \text{on } \partial A_\rho.
\]
For \(F \in C^{0,\alpha}(\overline{A_\rho})\), classical Schauder estimates \cite[\S6]{GilbargTrudinger2001} give
\[
  \|U\|_{C^{2,\alpha}(\overline{A_\rho})} \le C(\rho, R, \alpha)\|F\|_{C^{0,\alpha}(\overline{A_\rho})}.
\]
The constant \(C(\rho,R,\alpha)\) depends on the geometry of \(A_\rho\), which is controlled by \(\rho\) and \(R\) uniformly. In the bootstrap of Proposition~\ref{prop:annulus}, the right-hand side \(F = 1 + S_\eps\) has \(C^{0,\alpha}\) norm bounded by \(C(1 + \eps\|U_\eps\|_{C^{2,\alpha}})\), and the absorption argument yields \(\|U_\eps - u_0\|_{C^{2,\alpha}} \le C\eps\) with \(C\) independent of \(\eps\) but depending on \(\rho, R, \alpha\). The same estimates hold for the Dirichlet problem for \(-\Delta V_\eps\) with boundary data \(-c + O(\eps)\), yielding \(\|V_\eps - v_0\|_{C^{1,\alpha}} \le C\eps\).

% ===========================================================================

\medskip\noindent\emph{Laboratoire de Math\'ematiques, Universit\'e Choua\"ib Doukkali, El Jadida, Maroc}\\
\texttt{barkatou.m@ucd.ac.ma}, \texttt{elmorsalani@ucd.ac.ma}

\end{document}